\documentclass{amsart}

\usepackage{ae,aecompl}
\usepackage{graphicx}
\usepackage[all,cmtip]{xy}
\usepackage{amsmath, amscd}
\usepackage{amsthm}
\usepackage{amssymb}
\usepackage{amsfonts}
\usepackage{booktabs}
\usepackage{qsymbols}
\usepackage{latexsym}
\usepackage{mathrsfs}
\usepackage{cite}
\usepackage{color}
\usepackage{url}
\usepackage{enumerate}
\everymath{\displaystyle}
\usepackage{verbatim}
\usepackage[draft=false, colorlinks=true]{hyperref}

\usepackage[colorinlistoftodos,prependcaption,textsize=tiny]{todonotes}

\allowdisplaybreaks

\newcommand {\C}{{\mathbb C}}

\newcommand {\ud}{\mathrm{d}}
\newcommand {\ue}{e}

\newcommand {\Ell}{L}

\newcommand {\F}{{\mathcal{F}}}

\newcommand {\ui}{i}

\newcommand {\ind}{{\mathbf{1}}}

\newcommand {\La}{\mathcal{L}}
\newcommand {\calL}{\mathcal{L}}

\newcommand {\N}{\mathbb N}

\newcommand {\norm}[1]{\left\|#1\right\|}

\newcommand {\R}{\mathbb R}

\newcommand {\Sw}{\mathcal{S}}

\newcommand {\w}{\omega}

\newcommand {\Z}{\mathbb Z}

\newcommand {\vanish}[1]{\relax}

\newcommand{\wh}{\widehat}

\DeclareMathOperator{\Real}{Re}

\newtheorem{theorem}{Theorem}[section]
\newtheorem{lemma}[theorem]{Lemma}
\newtheorem{proposition}[theorem]{Proposition}

\theoremstyle{definition}

\newtheorem{remark}[theorem]{Remark}

\numberwithin{equation}{section}
\protected\def\ignorethis#1\endignorethis{}
\let\endignorethis\relax

\DeclareFontFamily{U}{mathx}{}
\DeclareFontShape{U}{mathx}{m}{n}{<-> mathx10}{}
\DeclareSymbolFont{mathx}{U}{mathx}{m}{n}
\DeclareMathAccent{\widecheck}{0}{mathx}{"71}

\author{Chenxi Deng}
\address{School of Mathematics and Statistics\\
Beijing Institute of Technology \\ Beijing 100081\\China }
\email{chenxideng@bit.edu.cn}

\author{Jan Rozendaal}
\address{Institute of Mathematics, Polish Academy of Sciences\\
\'{S}niadeckich 8\\
00-656 Warsaw\\
Poland}
\email{jrozendaal@impan.pl}

\author{Mark Veraar}
\address{Delft Institute of Applied
Mathematics\\
Delft University of Technology\\
P.O.~Box 5031\\
2628 CD Delft\\
The Netherlands}
\email{M.C.Veraar@tudelft.nl}

\title{Improved polynomial decay for unbounded semigroups}



\thanks{The first author is supported by China Scholarship Council (CSC). This research was funded in part by the National Science Center, Poland, grant 2021/43/D/ST1/00667. The second author is partially supported by NCN grant UMO-2023/49/B/ST1/01961. The third author is supported by the VICI subsidy VI.C.212.027 of the Netherlands Organization for Scientific Research (NWO)}

\begin{document}

\begin{abstract}
We obtain polynomial decay rates for $C_{0}$-semigroups, assuming 
that the resolvent grows polynomially at infinity in the complex right half-plane. Our results do not require the semigroup to be uniformly bounded, and for unbounded semigroups we improve upon previous results by, for example, removing a logarithmic loss on non-Hilbertian Banach spaces.
\end{abstract}

\maketitle

\section{Introduction}\label{sec:intro}

\subsection{Setting}\label{subsec:setting}

We study the asymptotic behavior of solutions to the abstract Cauchy problem
\begin{equation}\label{eq:ACP}
\begin{aligned}
\dot{u}(t)&=Au(t)\quad (t\geq0),\\
u(0)&=x,
\end{aligned}
\end{equation}
on a Banach space $X$. We assume that \eqref{eq:ACP} is well posed, so that the solution operators form a $C_{0}$-semigroup $(T(t))_{t\geq0}\subseteq\La(X)$ of bounded operators, with generator $A$. Throughout, we will consider $A$ satisfying $\overline{\C_{+}}\subseteq\rho(A)$, where $\C_{+}:=\{z\in\C\mid \Real(z)>0\}$ and $\rho(A):=\C\setminus\sigma(A)$ is the resolvent set of $A$. Under these assumptions, there are two well-known flavors of results that relate information about the resolvent operators $R(\lambda,A):=(\lambda-A)^{-1}$, $\lambda\in\rho(A)$, to asymptotic behavior of the semigroup orbits.

Firstly, the classical Gearhart--Huang--Pr\"{u}ss--Greiner theorem \cite{Gearhart78,Huang85,Pruss84} says that, if $X$ is a Hilbert space, then the semigroup $(T(t))_{t\geq0}$ is uniformly stable, and all orbits decay exponentially to zero, if and only if
\begin{equation}\label{eq:resunif}\sup_{\lambda\in\C_{+}}\|R(\lambda,A)\|_{\La(X)}<\infty.
\end{equation}
Versions of this theorem on non-Hilbertian Banach spaces were discovered later \cite{Weis95,Weis-Wrobel96,vanNeerven09}. Here an assumption such as \eqref{eq:resunif} typically guarantees exponential decay only for sufficiently smooth initial data, with the degree of smoothness depending on the geometry of the underlying Banach space. It is relevant to note that all these results make no a priori assumptions on the growth of the semigroup; only spectral information is required.

On the other hand, a more recent line of research considers the setting where the resolvent is not bounded on the right half-plane, but instead blows up along the imaginary axis at a specified rate. In this case the semigroup is not uniformly stable, and one can at best hope to obtain uniform decay rates for sufficiently smooth initial data. Semigroups with these properties arise naturally in the study of the damped wave equation
\begin{equation}\label{eq:damped}
\partial_{t}^{2}u(t,x)=\Delta_{g}u(t,x)-a(x)\partial_{t}u(t,x)\quad((t,x)\in \R\times M),
\end{equation}
on a Riemannian manifold $(M,g)$, where $a\in C(M)$ \cite{Ralston69,Rauch-Taylor74,Lebeau96,Burq98,Anantharaman-Leautaud14,ChSeTo20}. A succession of results in semigroup theory \cite{BaEnPrSc06,Batty-Duyckaerts08,Borichev-Tomilov10,BaChTo16,RoSeSt19} has elucidated the relationship between the rate of resolvent blowup and the rate of decay of classical solutions to \eqref{eq:ACP}, in the case where the semigroup is a priori assumed to be uniformly bounded. The latter assumption is in turn satisfied if the damping function $a$ in \eqref{eq:damped} is non-negative.

However, when considering functions $a$ in \eqref{eq:damped} that change sign, the associated semigroup need not be uniformly bounded and one may encounter unexpected spectral behavior (see e.g.~\cite{Renardy94,Rozendaal-Veraar18b}). 
Moreover, polynomially growing semigroups appear naturally in the analysis of Schr\"{o}dinger operators with unbounded potentials \cite{Davies-Simon91,Grigoryan06}, perturbed wave equations \cite{Goldstein-Wacker03,Paunonen14}, delay differential equations \cite{Sklyar-Polak17} and hyperbolic equations on non-Hilbertian Banach spaces \cite{Cordero-Nicola09,Rozendaal-Schippa23}.

Hence it is natural to wonder what can be said when one combines some of the difficulties of both the lines of research mentioned above, that is, if the semigroup $(T(t))_{t\geq0}$ is not uniformly bounded and the resolvent is not uniformly bounded on the right half-plane. This is the setting that will be considered in this article.

\subsection{Previous work}\label{subsec:previous}

Throughout, we suppose that $\C_{+}\subseteq\rho(A)$ and that there exist $\beta,C\geq 0$ such that
\[
\|R(\lambda,A)\|_{\La(X)} \leq C(1+|\lambda|)^{\beta}\quad(\lambda\in\C_{+}).
\]
It then follows that $\overline{\C_{+}}\subseteq\rho(A)$, but unless $\beta=0$, i.e.~unless \eqref{eq:resunif} holds, the resolvent might blow up along the imaginary axis, with polynomial rate at most $O(|\lambda|^{\beta})$. As in the work for uniformly bounded semigroups mentioned above, one hopes to derive polynomial rates of decay for semigroup orbits with sufficiently smooth initial data. 

In this regard, it was first shown in \cite{BaEnPrSc06} that, on general Banach spaces, for each $\rho\geq 0$ and $\tau>(\rho+1)\beta+1$ one has
\begin{equation}\label{eq:previous}
\|T(t)x\|_{X}\lesssim t^{-\rho}\|x\|_{D((-A)^{\tau})}
\end{equation}
for all $t\geq 1$ and $x\in D((-A)^{\tau})$. 
Later, \cite{Rozendaal-Veraar18a} improved this estimate under additional geometric assumptions on the underlying Banach space. Namely, if $X$ has Fourier type $p\in[1,2]$ (see Section \ref{subsec:Fourier}), then \eqref{eq:previous} holds for each $\tau>(\rho+1)\beta+\tfrac{1}{p}-\tfrac{1}{p'}$. Moreover, if $p=2$, i.e.~if $X$ is a Hilbert space, then one may let $\tau=(\rho+1)\beta$. However, it was left as an open question whether one may also let $\tau=(\rho+1)\beta+\tfrac{1}{p}-\tfrac{1}{p'}$ for Banach spaces with Fourier type $p\in[1,2)$ (see also \cite[Appendix B]{Rozendaal23}).

Recently, it was shown in \cite{Santana-Carvalho24} that the results from \cite{Rozendaal-Veraar18a} regarding \eqref{eq:previous} can in fact be improved. More precisely, for each $\rho>0$ and $\sigma>\tfrac{1}{p}-\tfrac{1}{p'}$ 
one has
\begin{equation}\label{eq:previous2}
\|T(t)x\|_{X}\lesssim t^{-\rho}\log(t)^{\sigma}\|x\|_{D((-A)^{\tau})}
\end{equation}
for $t\geq 2$ and $x\in D((-A)^{\tau})$, where $\tau=(\rho+1)\beta+\tfrac{1}{p}-\tfrac{1}{p'}$. That is, for $\rho>0$ and $p\in[1,2)$, \eqref{eq:previous2} attains the missing endpoint exponent from \cite{Rozendaal-Veraar18a}, up to a logarithmic loss. In fact, \cite{Santana-Carvalho24} combined methods from \cite{Rozendaal-Veraar18a} with ones from the theory for bounded semigroups in \cite{BaChTo16} and considered resolvents with more general growth behavior, but specializing to polynomially growing resolvents leads to \eqref{eq:previous2}.

Finally, it is important to emphasize that the results from \cite{Rozendaal-Veraar18a} and \cite{Santana-Carvalho24} are far from optimal if the semigroup $(T(t))_{t\geq0}$ is uniformly bounded. Indeed, in this case \cite{Batty-Duyckaerts08} yields, on general Banach spaces and for all $\rho\geq 0$,
\begin{equation}\label{eq:BD}
\|T(t)x\|_{X}\lesssim t^{-\rho}\log(t)^{\rho}\|x\|_{D((-A)^{\tau})}
\end{equation}
for $t\geq 2$ and $x\in D((-A)^{\tau})$, where $\tau=\rho\beta$. Moreover, by \cite{Borichev-Tomilov10}, if $X$ is a Hilbert space then the logarithmic factor in \eqref{eq:BD} can be removed, yielding \eqref{eq:previous} for $\tau=\rho\beta$. 
On the other hand, for unbounded semigroups on Hilbert spaces and for $\rho=0$ one cannot in general expect to obtain \eqref{eq:previous} for $\tau<(\rho+1)\beta$, as follows from an example of Wrobel (see \cite[Example 4.1]{Wrobel89} and \cite[Example 4.20]{Rozendaal-Veraar18a}). We also refer to \cite[Section 4.7.1]{Rozendaal-Veraar18a} for an application to polynomially growing semigroups of the combination of \eqref{eq:BD} and a rescaling argument.

\subsection{Main result}\label{subsec:main}

For $\tau>0$ and $q\in[1,\infty]$, we will work with the real interpolation space $D_{A}(\tau,q):=(X,D(A^{m}))_{\tau/m,q}$, where $m\in\N$ with $m>\tau$ is arbitrary (see also \eqref{eq:interspace}). Moreover, we refer to \eqref{eq:HLtype} and \eqref{eq:HLcotype} for the definitions of Hardy--Littlewood type and Hardy--Littlewood cotype, respectively. The following is our main result.

\begin{theorem}\label{thm:main}
Let $A$ be the generator of a $C_{0}$-semigroup $(T(t))_{t\geq0}$ on a Banach space $X$ with Fourier type $p\in[1,2]$. Suppose that $\C_{+}\subseteq\rho(A)$, and that there exist 
$\beta>0$ and 
$C\geq0$ such that
\begin{equation}
\label{eq:mainas}
\|R(\lambda,A)\|_{\La(X)} \leq C(1+|\lambda|)^{\beta}\quad(\lambda\in\C_{+}).
\end{equation}
Let $\rho\geq0$ and set $\tau:=(\rho+1)\beta+\tfrac{1}{p}-\tfrac{1}{p'}$. Then there exists a $C_{\rho}\geq0$ such that \begin{equation}\label{eq:main}
\|T(t)x\|_{X}\leq C_{\rho}t^{-\rho}\|x\|_{D_A(\tau,p)}
\end{equation}
for all $t\geq1$ and $x\in D_{A}(\tau,p)$.
If $\rho>0$, then \eqref{eq:main} also holds with $D_A(\tau,p)$ replaced by $D_A(\tau,q)$ for any $q\in [1, \infty]$, or by $D((-A)^{\tau})$.

Suppose, additionally, that $p>1$ and that $X$ has Hardy--Littlewood type $p$ or Hardy--Littlewood cotype $p'$. Then, for $\rho=0$, \eqref{eq:main} also holds with $D_{A}(\tau,p)$ replaced by $D((-A)^{\tau})$.
\end{theorem}

The first two statements of Theorem \ref{thm:main} are contained in the main text as Theorem \ref{thm:polinter}, while the last statement is Theorem \ref{thm:polydecayHL}.

Given that any Banach space has Fourier type $p=1$, the first part of Theorem \ref{thm:main} applies to general Banach spaces.  For $p\in(1,2]$, the assumptions on $X$ in Theorem \ref{thm:main} are satisfied in particular if $X$ is isomorphic to a closed subspace of $L^{r}(\Omega)$, for $\Omega$ a measure space and $r=p$ or $r=p'$ (see Section \ref{subsec:Fourier}). 

For $p\in[1,2)$, the first part of Theorem \ref{thm:main} improves \eqref{eq:previous2} by removing the logarithmic factor for $\rho>0$, and it yields an endpoint result for $\rho=0$. 
The second part of Theorem \ref{thm:main} in turn fully extends \eqref{eq:previous} to $\rho=0$ and $\tau=\beta+\tfrac{1}{p}-\tfrac{1}{p'}$, under additional geometric assumptions. Also note that, for all $p\in[1,2]$ and $\rho>0$, \eqref{eq:main} involves a larger space of initial data than considered in \cite{Rozendaal-Veraar18a} and \cite{Santana-Carvalho24}, since $D((-A)^{\tau})\subseteq D_{A}(\tau,\infty)$. On the other hand, for $\rho=0$, \eqref{eq:main} complements the main result of \cite{Rozendaal-Veraar18a} on Hilbert spaces, since in general one neither has $D((-A)^{\tau})\subseteq D_{A}(\tau,2)$ nor $D_{A}(\tau,2)\subseteq D((-A)^{\tau})$.

The exponent $\tau$ in Theorem \ref{thm:main} is sharp for $p=2$ and $\rho=0$, as noted above, and for general $p\in[1,2]$ as $\beta\to0$, as follows from a modification of an example of Arendt concerning exponential stability (see \cite[Example 5.1.11]{ArBaHiNe11} and \cite[Section 4]{Weis-Wrobel96}). We do not know whether, for a general Banach space with Fourier type $p\in[1,2)$ and for $\rho=0$, \eqref{eq:main} also holds with $D_{A}(\tau,p)$ replaced by $D((-A)^{\tau})$.

For any $C_{0}$-semigroup $(T(t))_{t\geq0}$ there exists an $\w\in\R$ such that
\begin{equation}\label{eq:semigroupgrowth}
\|T(t)\|_{\La(X)}\lesssim e^{\w t}\quad(t\geq0).
\end{equation}
As already noted, only the case $\w>0$ will be of interest in this article. However, it follows from \eqref{eq:semigroupgrowth} that \eqref{eq:mainas} holds whenever $\Real(\lambda)\geq\w_{0}$, for $\w_{0}>\w$. Moreover, \eqref{eq:mainas} directly extends to $\lambda\in i\R$ as well.
Hence \eqref{eq:mainas} is in fact an assumption on the growth of the resolvent as $\lambda$ tends to infinity in the strip $\{\lambda\in\C\mid 0\leq  \Real(\lambda)\leq \w_{0}\}$.  

One may weaken assumption \eqref{eq:mainas} somewhat, by requiring instead that
\begin{equation}\label{eq:mainweak1}
\|R(\lambda,A)\|_{\La(X)} \lesssim (1+|\lambda|)^{\beta_{0}}\quad(\lambda\in\C_{+})
\end{equation}
for \emph{some} $\beta_{0}>0$, and that
\[
\|R(i\xi,A)\|_{\La(X)} \lesssim (1+|\xi|)^{\beta}\quad(\xi\in\R).
\]
Then the conclusion of Theorem \ref{thm:main} still holds, and the specific value of $\beta_{0}$ in \eqref{eq:mainweak1} plays no role. Indeed, the place in the proof of Theorem \ref{thm:main} where one genuinely uses polynomial resolvent bounds for $\lambda\in\C_{+}$ is in the proofs of Theorems \ref{thm:abstractpol} and \ref{thm:polydecayHL}, to obtain a dense subset of initial values for which the semigroup orbits are integrable,
and there the value of $\beta_{0}$ is irrelevant. Instead, as in the theory for uniformly bounded semigroups, to obtain concrete rates of decay we work with the behavior of the resolvent on the imaginary axis. 

As in \cite{Rozendaal-Veraar18a,Santana-Carvalho24}, our techniques in principle also allow for $A$ to have a singularity at zero. More precisely, one could suppose that \eqref{eq:mainas} holds for $|\lambda|\geq1$, and that there exists an $\alpha>0$ such that $\|R(\lambda,A)\|_{\La(X)} \lesssim |\lambda|^{-\alpha}$ for $|\lambda|<1$. In this case one has to assume additionally that $-A$ is an injective sectorial operator (see Remark \ref{rem:secauto}), and the initial values have to be restricted to the range of a suitable fractional power of $-A$. For simplicity, we will not consider such a setting in this article.

\subsection{The strategy of the proof}

Our approach is similar to that in \cite{Rozendaal-Veraar18a} (see also \cite{Rozendaal23}), applying Fourier multiplier theory to the resolvent on the imaginary axis. However, whereas \cite{Rozendaal-Veraar18a} mostly involved Fourier multipliers from $L^{p}(\R;Y)$ to $L^{q}(\R;X)$ for suitable $1\leq p\leq q\leq \infty$ and $Y\subseteq X$, in the present article we proceed differently.

Namely, the first part of Theorem \ref{thm:main} is proved using  Proposition \ref{prop:multFourtype}, which considers multipliers between the Besov space $B^{s}_{p,p}(\R;Y)$ and $L^{p'}(\R;X)$, for suitable values of $p$ and $s$. Working with such multipliers allows us to obtain endpoint estimates. In turn, Besov spaces are intimately connected to the real interpolation method, and in Proposition \ref{prop:resdecay} we show that real interpolation spaces can also be used effectively to cancel out resolvent growth, as is required to satisfy the conditions of our Fourier multiplier theorems. This somewhat different approach also necessitates other changes to the setup from \cite{Rozendaal-Veraar18a}.

On the other hand, for the second part of Theorem \ref{thm:main} we consider Fourier multipliers between weighted spaces $L^{p}(\R,w;Y)$ and $L^{q}(\R,v;X)$, for suitable weights $w$ and $v$. This setting allows us to obtain endpoint results involving fractional domains, at the cost of having to make a priori assumptions about the mapping properties of the Fourier transform between such weighted spaces.

\subsection{Organization}\label{subsec:organize}

In Section \ref{sec:preliminaries} we collect some preliminaries on the vector-valued Fourier transform, vector-valued Besov spaces, and interpolation spaces associated with semigroup generators. In Section \ref{sec:stabinter} we then prove the first part of Theorem \ref{thm:main}, and in Section \ref{sec:stabfrac} we prove the final statement in Theorem \ref{thm:main}.

\subsection{Notation and terminology}\label{subsec:notation}

The natural numbers are $\N=\{1,2,\ldots\}$, and $\N_{0}:=\N\cup\{0\}$. We write $\C_{+}:=\{z\in\C\mid \Real(z)>0\}$ for the open complex right half-plane.

For $p\in[1,\infty]$ and $w:\R\to [0,\infty)$ measurable, we denote by $L^{p}(\R,w;X)$ the Bochner space of equivalence classes of strongly measurable, $p$-integrable, $X$-valued functions on $\R$ with respect to the weight $w$, endowed with the norm
\[
\|f\|_{L^{p}(\R,w;X)}:=\Big(\int_{\R}\|f(x)\|_{X}^{p}w(x)\ud x\Big)^{1/p}
\]
for $f\in L^{p}(\R,w;X)$. We simply denote this space by $L^{p}(\R;X)$ when $w\equiv 1$. For $\gamma\in\R$, the weight $w_{\gamma}:\R\to[0,\infty)$ is given by
\begin{equation}\label{eq:weight}
w_{\gamma}(x):=|x|^{\gamma}\quad(x\in\R).
\end{equation}
The H\"{o}lder conjugate $p'\in[1,\infty]$ of $p\in[1,\infty]$ is defined by $1=\tfrac{1}{p}+\tfrac{1}{p'}$. We write $\ind_{S}$ for the indicator function of a set $S$.

The space of bounded operators between complex Banach spaces $X$ and $Y$ is $\La(X,Y)$, and $\La(X):=\La(X,X)$. 
The domain of a closed operator $A$ on $X$ is $D(A)$. 

We use the notation $f(s)\lesssim g(s)$ 
to indicate that $f(s)\leq Cg(s)$ for all $s$ and a constant $C\geq0$ independent of $s$, and similarly for $f(s)\gtrsim g(s)$ and $f(s)\eqsim g(s)$.
\section{Preliminaries}\label{sec:preliminaries}

In this section we first collect some basic definitions involving the vector-valued Fourier transform, and then we introduce Besov spaces and state two results which will be needed to prove the first half of Theorem \ref{thm:main}. Finally, we collect background on interpolation spaces and we prove two key results about them.

\subsection{The Fourier transform}\label{subsec:Fourier}

Let $X$ be a Banach space. The class of $X$-valued Schwartz functions on $\R$ is denoted by $\Sw(\R;X)$, and the space of $X$-valued tempered distributions by $\Sw'(\R;X)$. The Fourier transform of $f\in\Sw'(\R;X)$ is denoted by $\F f$ or $\widehat{f}$. If $f\in\Ell^{1}(\R;X)$ then
\[
\F f(\xi)=\int_{\R}\ue^{-\ui \xi t}f(t)\,\ud t\quad (\xi\in\R).
\]
One says that $X$ has \emph{Fourier type} $p\in[1,2]$ if $\F:L^p(\R;X)\to L^{p'}(\R;X)$ is bounded. 
Every Banach space $X$ has Fourier type $1$, and $X$ has Fourier type $2$ if and only if $X$ is isomorphic to a Hilbert space (see \cite{Kwapien72}).

We say that $X$ has \emph{Hardy--Littlewood type} $p\in(1,2]$ if
\begin{equation}\label{eq:HLtype}
\F:L^{p}(\R;X)\to L^{p}(\R,w_{p-2};X)
\end{equation}
is bounded, where $w_{\gamma}$ is defined in \eqref{eq:weight} for $\gamma\in\R$. Moreover, $X$ has \emph{Hardy--Littlewood cotype} $q\in[2,\infty)$ if
\begin{equation}\label{eq:HLcotype}
\F:L^{q}(\R,w_{q-2};X)\to L^{q}(\R;X)
\end{equation}
is bounded. Note that, if $X=\C$, then \eqref{eq:HLtype} is the Hardy--Littlewood inequality. In the latter case, and in fact for any Hilbert space $X$, \eqref{eq:HLtype} holds for all $p\in(1,2]$, and \eqref{eq:HLcotype} for all $q\in[2,\infty)$.

If $X$ has Fourier type $p_{0}\in(1,2]$, then $X$ has Hardy--Littlewood type $p$ for all $p\in(1,p_{0})$, and Hardy--Littlewood cotype for all $q\in(p_{0}',\infty)$ (see {\cite[Proposition 3.5]{Dominguez-Veraar21}). Also, if $X$ is a Banach lattice which is $p$-convex and $p$-concave with $p\in(1,\infty)$, then $X$ has Fourier type $p$ and Hardy--Littlewood type $p$ if $p\leq 2$, and Fourier type $p'$ and Hardy--Littlewood cotype $p$ if $p\geq 2$ (see \cite[Proposition 2.2]{GaToKa96} and \cite[Proposition 6.9]{Dominguez-Veraar21}). This holds in particular if $X$ is isomorphic to a closed subspace of $L^{p}(\Omega)$, for $\Omega$ any measure space. For more on the relation between the notions of Fourier type, Hardy--Littlewood (co)type, and convexity and concavity in Banach lattices, we refer to \cite{Dominguez-Veraar21}.

Let $Y$ be a Banach space and $m:\R\to \La(Y,X)$. We say that $m$ is \emph{$X$-strongly measurable} if $\xi\mapsto m(\xi)y$ is a strongly measurable $X$-valued map for every $y\in Y$. In this article we will consider $m$ which have the additional property that there exist $\alpha,C_{\alpha}\geq 0$ 
such that $\|m(\xi)\|_{\La(Y,X)}\leq C_{\alpha}(1+|\xi|)^{\alpha}$ for all $\xi\in\R$. In this case, we may set
\[
T_{m}f := \mathcal{F}^{-1}(m\widehat{f}\,)\quad (f\in  \mathcal{S}(\R;X)).
\]
Then  $T_m: \mathcal{S}(\R;X)\to  \mathcal{S}'(\R;Y)$ is the \emph{Fourier multiplier operator} with \emph{symbol} $m$.

\subsection{Besov spaces} \label{subsec:Besov}

Throughout this article, fix an inhomogeneous Littlewood--Paley sequence $(\phi_k)_{k\in\N_{0}}\subseteq C^{\infty}_{c}(\R)$. That is, one has $\phi_{1}(\xi)=0$ if $|\xi|\notin [1/2,2]$, $\phi_{k}(\xi)=\phi_{1}(2^{-k+1}\xi)$ for each $k>1$ and $\xi\in\R$, and
\[
\sum_{k=0}^{\infty}\phi_{k}(\xi)=1\quad(\xi\in\R).
\]
Let $p,q\in[1, \infty]$ and $s\in \R$. Then the Besov space $B^s_{p,q}(\R;X)$
consists of all $f \in \mathcal{S}'(\R;X)$ such that $\F^{-1}(\phi_k)* f\in L^p(\R;X)$ for each $k \ge 0$, and such that
\[
\|f\|_{B^s_{p,q}(\R;X)}:=\|(2^{ks}\F^{-1}(\phi_k)* f)_{k\ge 0}\|_{\ell^q(L^p(\R;X))}<\infty.
\]
Then $\mathcal{S}(\R;X)\subseteq B^s_{p,q}(\R;X)$
, by \cite[Proposition 14.4.3]{HyNeVeWe23}. Moreover, if $p,q<\infty$, then $\mathcal{S}(\R;X)$ is a dense subset of $B^s_{p,q}(\R;X)$. Finally, we will use the simple observation that
\begin{equation}\label{eq:Besemb1}
B^{s}_{p,q}(\R;X)\subseteq B^{r}_{p,1}(\R;X)
\end{equation}
for all $p,q\in[1,\infty]$ and $s,r\in\R$ with $s>r$, and that
\begin{equation}\label{eq:Besemb2}
B^0_{p,1}(\R;X)\subseteq L^p(\R;X)\subseteq B^0_{p,\infty}(\R;X)
\end{equation}
for all $p\in[1,\infty]$ ({see e.g.}~\cite[Proposition 14.4.18]{HyNeVeWe23}).

The following lemma will be used in the proof of Proposition \ref{prop:interpsemigroup}.

\begin{lemma}\label{lem:pointwise}
Let $p\in [1, \infty)$, $q\in [1, \infty]$ and $s\in(0,1/p)$. Then there exists a $C\geq0$ such that $\ind_{(0,\infty)} f\in B^{s}_{p,q}(\R;X)$ for all $f\in B^{s}_{p,q}(\R;X)$, and
\[
\|\ind_{(0,\infty)}f\|_{B^{s}_{p,q}(\R;X)}\leq C \|f\|_{B^{s}_{p,q}(\R;X)}.
\]
\end{lemma}
\begin{proof}
For $p>1$, the statement in fact holds for $s\in(-1/p,1/p)$, as is shown in \cite[Corollary 14.6.35]{HyNeVeWe23}. In the proof of the latter result, one can see that for $s\in(0,1/p)$ one may also allow $p=1$.
\end{proof}

Finally, the following Fourier multiplier result, \cite[Proposition 14.5.7]{HyNeVeWe23}, is one of the key ingredients in the proof of the first half of Theorem \ref{thm:main}.

\begin{proposition}\label{prop:multFourtype}
Let $X$ and $Y$ be Banach spaces with Fourier type $p\in[1,2]$, and let $m:\R\to\La(Y,X)$ be $X$-strongly measurable, with $\sup_{\xi\in\R}\|m(\xi)\|_{\La(Y,X)}<\infty$. Then
\[
T_{m}:B^{1/p-1/p'}_{p,p}(\R;Y)\to L^{p'}(\R;X)
\]
is bounded.
\end{proposition}

\subsection{Interpolation spaces}\label{subsec:inter}

Let $A$ be a linear operator on a Banach space $X$.
For $\omega\in (0,\pi)$, set $S_{\omega}:=\{z\in\C\setminus\{0\}\mid |\arg(z)|<\w\}$.
Then $-A$ is a \emph{sectorial operator} if there exists an $\omega\in(0,\pi)$ such that $\sigma(-A)\subseteq \overline{S_{\omega}}$, and
\begin{equation}\label{eq:sectorial}
\sup \{\|\lambda (\lambda+A)^{-1}\|_{\mathcal{L}(X)}\,|\,\lambda\in \C\backslash\overline{S_{\omega'}}\}<\infty
\end{equation}
for each $\omega'\in (\omega,\pi)$. If $-A$ is a sectorial operator, then 
the fractional power $(-A)^{\alpha}$ is well defined for each $\alpha\in\C_{+}$, cf.~\cite[Chapter 3]{Haase06a}. If, additionally, $A$ is injective, then $(-A)^{\alpha}$ is well defined for all $\alpha\in\C$. Note that $D((-A)^{\beta})\subseteq D((-A)^{\alpha})$ whenever $\beta\in\C$ satisfies $\Real(\beta)>\Real(\alpha)$.

\begin{remark}\label{rem:secauto}
Throughout this article, as in Theorem \ref{thm:main}, we will consider $C_{0}$-semigroups $(T(t))_{t\geq0}$ with generator $A$ such that $\C_{+}\subseteq\rho(A)$ and
\begin{equation}\label{eq:secauto}
\|R(\lambda,A)\|_{\La(X)} \leq C(1+|\lambda|)^{\beta}\quad(\lambda\in\C_{+}),
\end{equation}
where $\beta,C\geq0$ are independent of $\lambda$. Under these assumptions, $-A$ is a sectorial operator of angle $\pi/2$. Indeed, the semigroup generation property implies that $\|R(\lambda,A)\|_{\La(X)}\lesssim 1/\Real(\lambda)$ for $\Real(\lambda)$ large (as follows from \eqref{eq:semigroupgrowth}), which gives a uniform bound in \eqref{eq:sectorial} for $
|\lambda|$ large if $\w'>\pi/2$. On the other hand, \eqref{eq:secauto} implies that $\overline{\C_{+}}\subseteq \rho(A)$, which in turn yields the required bound in \eqref{eq:sectorial} for $|\lambda|$ small.
\end{remark}

Let $-A$ be a sectorial operator on a Banach space $X$, and let $\tau\in(0,\infty)$ and $q\in [1,\infty]$. Then the real interpolation space associated with $A$, $\tau$ and $q$ is
\begin{equation}\label{eq:interspace}
D_A(\tau,q):=(X, D((-A)^{\alpha}))_{\tau/\alpha,q},
\end{equation}
where $\alpha\in(\tau,\infty)$ is arbitrary. It follows from reiteration that $D_A(\tau,q)$ is independent of the choice of $\alpha$. In particular, one has
\[
D_A(\tau,q)=(X, D((-A)^{m}))_{\tau/m,q}=(X, D(A^{m}))_{\tau/m,q}
\]
whenever $m\in\N$ satisfies $m>\tau$. By basic properties of interpolation spaces, $D_{A}(\tau,q)\subseteq D_{A}(\sigma,r)$ if $\sigma<\tau$, or if $\sigma= \tau$ and $r\geq q$. By \cite[Corollary 6.6.3]{Haase06a},
\begin{equation}\label{eq:interfrac}
D_A(\tau,1)\subseteq D((-A)^{\tau})\subseteq D_A(\tau,\infty)
\end{equation}
for all $\tau>0$. Also, $D((-A)^{\alpha})$ is a dense subset of $D_A(\tau,q)$ for all $\alpha>\tau$ and $q<\infty$, by \cite[Theorem 6.6.1]{Haase06a}.

Finally, if $X$ has Fourier type $p\in[1,2]$ and $A$ is injective, then both $D((-A)^{\tau})$ and $D_A(\tau,p)$ also have Fourier type $p$, for all $\tau>0$. Indeed, 
$(-A)^{\tau}:D((-A)^{\tau})\to X$ is an isomorphism, while for $D_A(\tau,p)$ the statement follows from \eqref{eq:interspace} by interpolation (see also \cite[Proposition 2.4.17]{HyNeVeWe16}).

The following proposition, connecting interpolation spaces to the Besov spaces from the previous subsection, will play a key role in the proof of part of Theorem \ref{thm:abstractpol}.

\begin{proposition}\label{prop:interpsemigroup}
Let $A$ be the generator of a $C_{0}$-semigroup $(T(t))_{t\geq 0}$ on a Banach space $X$, and suppose that $-A$ is a sectorial operator. Let $M\geq0$ and $\w\in\R$ be such that $\|T(t)\|_{\La(X)} \leq Me^{(\omega-1) t}$ for all $t\geq 0$, and let $p\in [1, \infty)$, $q\in [1, \infty]$ and $s\in (0,1/p)$. Then there exists a constant $C\geq0$ such that $[t\mapsto \ind_{(0,\infty)}(t)e^{-\omega t}T(t) x]\in B^s_{p,q}(\R;X)$ for all $x\in D_A(s,q)$, with
\[
\|[t\mapsto \ind_{(0,\infty)}(t)e^{-\omega t}T(t) x]\|_{B^s_{p,q}(\R;X)}\leq C\|x\|_{D_A(s,q)}.
\]
\end{proposition}
\begin{proof}
Let $J:X\to L^p(\R;X)$ be the bounded linear operator given by $Jx(t) :=  e^{-\omega |t|}T(|t|) x$, for $x\in X$ and $t\in\R$. Since $(Jx)'(t) =-\text{sign}(t) e^{-\omega |t|}T(|t|) (\omega-A)x$ for $x\in D(A)$ and $t\neq0$, the restricted operator $J:D(A)\to W^{1,p}(\R;X)$ is bounded. Real interpolation (see \cite[Theorem 14.4.31]{HyNeVeWe23}) shows that $J:D_A(s,q)\to B^{s}_{p,q}(\R;X)$ is bounded as well. Now the proof is concluded by applying Lemma \ref{lem:pointwise}.
\end{proof}

In turn, the following proposition will be crucial for the proof of Theorem \ref{thm:main}.

\begin{proposition}\label{prop:resdecay}
Let $A$ be be the generator of a $C_{0}$-semigroup $(T(t))_{t\geq 0}$ on a Banach space $X$. Suppose that $\C_{+}\subseteq\rho(A)$, 
and that there exist $\beta>0$ and $C\geq0$ such that
\begin{equation}\label{eq:resdecay}
\|R(\lambda,A)\|_{\La(X)} \leq C(1+|\lambda|)^{\beta}\quad(\lambda\in\C_{+}).
\end{equation}
Then $i\R\subseteq\rho(A)$, and
\[
\sup\{\|R(i\xi,A)^{k}\|_{\La(D_A((n+1)\beta,q),X)}\mid \xi\in \R, k\in\{0,\ldots, n+1\}\}<\infty
\]
for all $n\in\N_{0}$ and $q\in [1,\infty]$.
\end{proposition}
\begin{proof}
The required statement is trivial for $k=0$ and
$|\xi|<1$, so henceforth we will consider $k>0$ and $\xi\in\R$ with $|\xi|\geq1$. 

By basic properties of resolvents, it follows from \eqref{eq:resdecay} that  $i\R\subseteq\rho(A)$ and
\begin{equation}\label{eq:resdecayextend}
\|R(i\xi,A)\|_{\La(X)} \leq C(1+|\xi|)^{\beta}\quad(\xi\in\R).
\end{equation}
As noted in Remark \ref{rem:secauto}, $-A$ is a sectorial operator. Moreover, $D((-A)^{\alpha})=D((1-A)^{\alpha})$ for any $\alpha>0$, by \cite[Proposition 15.2.12]{HyNeVeWe23}. Hence, combining \eqref{eq:resdecayextend} and \cite[Proposition 3.4]{Rozendaal-Veraar18a}  yields
\begin{equation}\label{eq:resdecay1}
\sup_{|\xi|\geq1}\|R(i\xi,A)(-A)^{-\beta}\|_{\mathcal{L}(X)}<\infty.
\end{equation}
Then, for $k\in\{1,\ldots,n\}$ and $|\xi|\geq1$,
\begin{align*}
\|R(i\xi,A)^{k}\|_{\La(D((-A)^{n\beta}),X)}&\eqsim \|R(i\xi,A)^{k}(-A)^{-n\beta}\|_{\mathcal{L}(X)}\\
&\leq \|R(i\xi,A)^{k}(-A)^{-k\beta}\|_{\La(X)}\|(-A)^{-(n-k)\beta}\|_{\La(X)}\lesssim 1.
\end{align*}
Together with \eqref{eq:resdecayextend}, this implies
\begin{equation}\label{eq:resdecay2}
\|R(i\xi,A)^{k}\|_{\La(D((-A)^{n\beta}),X)} \lesssim (1+|\xi|)^{\beta}\quad  (k\in\{1,\ldots, n+1\}).
\end{equation}
On the other hand, another application of \cite[Proposition 3.4]{Rozendaal-Veraar18a} shows that
\[
\|R(i\xi,A)(-A)^{-\beta-1}\|_{\mathcal{L}(X)}\lesssim(1+|\xi|)^{-1}.
\]
This, combined with \eqref{eq:resdecay1}, yields
\begin{equation}\label{eq:resdecay3}
\begin{aligned}&\|R(i\xi,A)^{k}\|_{\La(D((-A)^{(n+1)\beta+1}),X)} \eqsim \|R(i\xi,A)^{k}(-A)^{-(n+1)\beta-1}\|_{\mathcal{L}(X)}\\
&\leq \|R(i\xi,A)(-A)^{-\beta}\|^{k-1}_{\La(X)}\|R(i\xi,A)(-A)^{-\beta-1}\|_{\La(X)}\|(-A)^{-\beta}\|_{\La(X)}^{n-k+1}\\
   &\lesssim (1+|\xi|)^{-1},
\end{aligned}
\end{equation}
for all $k\in\{1,\ldots, n+1\}$ and $|\xi|\geq1$.

Now, by \eqref{eq:interfrac}, \eqref{eq:resdecay2} and \eqref{eq:resdecay3}, we have
\begin{align*}
\|R(i\xi,A)^{k}\|_{\calL(D_A(n\beta,1),X)} &\lesssim (1+|\xi|)^{\beta},\\
\|R(i\xi,A)^{k}\|_{\calL(D_A((n+1)\beta+1,1),X)} &\lesssim(1+|\xi|)^{-1},
\end{align*}
for $|\xi|\geq 1$ and $ k\in\{1,\ldots, n+1\}$. Finally, since
\[
D_{A}((n+1)\beta,q)=(D_A(n\beta,1),D_A((n+1)\beta+1,1))_{\frac{\beta}{1+\beta},q},
\]
interpolating these estimates yields $\sup_{|\xi|\geq1}\|R(i\xi,A)^{k}\|_{\La(D_{A}((n+1)\beta,q),X)}<\infty$.
\end{proof}

\section{Polynomial stability on real interpolation spaces}\label{sec:stabinter}

This section is devoted to the proof of the first half of Theorem \ref{thm:main}. To this end, we need two preliminary results.

We first require the following extension of \cite[Proposition 3.2]{Rozendaal-Veraar18a} to the mixed Besov-Lebesgue setting. 
\begin{proposition}\label{prop:qtoinf}
Let $A$ be the generator of a $C_{0}$-semigroup $(T(t))_{t\geq0}$ on a Banach space $X$, and let $Y$ be a Banach space that is continuously embedded in $X$. Suppose that $i\R\subseteq\rho(A)$, and that there exist $\beta>0$ and $C\geq0$ such that
\begin{equation}\label{eq:qtoinfas1}
\|R(i\xi,A)\|_{\La(Y,X)}\leq C(1+|\xi|)^{\beta}\quad(\xi\in\R).
\end{equation}
Let $p\in[1,\infty)$ and $s\in[0,\infty)$ be such that either $s>0$, or $s=0$ and $p=1$, and suppose that there exist $q\in[1,\infty]$ and $n\in\N$ such that
\begin{equation}\label{eq:qtoinfas2}
T_{R(i\cdot,A)^j}:B^{s}_{p,p}(\R;Y)\to L^q(\R;X)
\end{equation}
is bounded for each $j\in \{n-1, n\}\cap\N$.
Then
\[
T_{R(\ui\cdot,A)^{n}}:B^{s}_{p,p}(\R;Y)\to L^{\infty}(\R;X)
\]
is bounded.
\end{proposition}
We only assume \eqref{eq:qtoinfas1} to guarantee that the Fourier multiplier operator in \eqref{eq:qtoinfas2} is 
well defined; the specific value of $\beta$ plays no role here. 

\begin{proof}
The proof is analogous to that of \cite[Proposition 3.2]{Rozendaal-Veraar18a}. For the convenience of the reader, we provide the argument. It suffices to show that there exists a $C\geq0$ such that
\begin{equation}\label{eq:qtoinf0}
\sup_{k\leq \sigma\leq k+1}\|T_{R(i\cdot,A)^n}f(\sigma)\|_{X}\leq C\|f\|_{B^s_{p,p}(\mathbb{R};Y)}
\end{equation}
for every $f\in\Sw(\R;Y)$ and $k\in\Z$, since $\Sw(\R;Y)$ is a dense subset of $B^{s}_{p,p}(\R;Y)$.

For each $j\in\{n-1,n\}\cap\N$ there exists a $K_{j}\geq0$ independent of $f$ such that
\begin{equation}\label{eq:qtoinf1}
\|T_{R(i\cdot,A)^j}f\|_{L^q(\mathbb{R};X)}\leq K_j \|f\|_{B^s_{p,p}(\mathbb{R};Y)}.
\end{equation}
Hence there exists a $t\in [k-1,k]$ such that
\begin{equation}\label{eq:qtoinf2}
\|T_{R(i\cdot,A)^j}f(t)\|_X\leq  K_j \|f\|_{B^s_{p,p}(\R;Y)}.
\end{equation}
Now let $\tau\in [0,2]$. One can check that
\begin{equation}\label{eq:qtoinf3}
T(\tau)T_{R(i\cdot,A)^n}f(t) =T_{R(i\cdot,A)^n}f(t+\tau)-\int_0^\tau T(r)T_{R(i\cdot,A)^{n-1}} f(t+\tau-r) \ud r.
\end{equation}
Hence, by \eqref{eq:qtoinf2}, H{\"o}lder's inequality and \eqref{eq:qtoinf1}, for $n>1$ one has
\begin{align*}
\|T_{R(i\cdot,A)^n}f(t+\tau)\|_{X}&\lesssim\|T_{R(i\cdot,A)^n}f(t)\|_{X}+\int_0^\tau\|T_{R(i\cdot,A)^{n-1}}f(t+\tau-r)\|_{X}\ud r\\
         &\leq K_n\|f\|_{B^s_{p,p}(\mathbb{R};Y)}+\tau^{1/{q'}}\|T_{R(i\cdot,A)^{n-1}}f\|_{L^q(\R;X)}\\
         &\lesssim\|f\|_{B^s_{p,p}(\mathbb{R};Y)}.
\end{align*}
This implies \eqref{eq:qtoinf0} for $n>1$.

Finally, for $n=1$, by the assumptions on $p$ and $s$ as well as \eqref{eq:Besemb1} and \eqref{eq:Besemb2}, one has $B^s_{p,p}(\R;Y)\subseteq B^0_{p,1}(\mathbb{R};Y)\subseteq L^p(\R;Y)$. Hence H\"{o}lder's inequality gives
\[
\int_0^\tau \|f(t+\tau-r)\|_X \ud r\lesssim \int_0^\tau \|f(t+\tau-r)\|_Y \ud r\leq \tau^{1/p'}\|f\|_{L^{p}(\R;Y)}\lesssim \|f\|_{B^{s}_{p,p}(\R;Y)}.
\]
By combining this with \eqref{eq:qtoinf3} in the same manner as before, one obtains \eqref{eq:qtoinf0}.
\end{proof}

We will also rely on the following version of \cite[Theorem 4.6]{Rozendaal-Veraar18a} in the mixed Besov-Lebesgue setting.

\begin{theorem}\label{thm:abstractpol}
Let $A$ be the generator of a $C_{0}$-semigroup $(T(t))_{t\geq 0}$ on a Banach space $X$. Suppose that $\C_{+}\subseteq\rho(A)$, and that there exist $\beta>0$ 
and $C\geq0$ 
such that
\begin{equation}
\label{eq:abstractpolas11}\|R(\lambda,A)\|_{\La(X)}
\leq C(1+|\lambda|)^{\beta}\quad(\lambda\in\C_{+}).
\end{equation}
Let $\gamma>0$, $p\in[1,\infty)$ and $s\in(0,1/p)$, and suppose that there exist $n\in\N_{0}$ and $q\in[1,\infty]$ such that
\begin{equation}\label{eq:abstractpolas2}
T_{R(i\cdot,A)^{k}}:B^{s}_{p,p}(\R;D_A(\gamma,p))\to L^q(\R;X)
\end{equation}
is bounded for each $k\in\{n-1,n,n+1\}\cap\N$. Then there exists a $C_{n}\geq0$ such that
\begin{equation}\label{eq:abstractpol}
\|T(t)x\|_{X}\leq C_{n}t^{-n}\|x\|_{D_A(\gamma+s,p)}
\end{equation}
for all $t\geq1$ and $x\in D_A(\gamma+s,p)$.
\end{theorem}
Note that \eqref{eq:abstractpolas11} automatically extends to all $\lambda\in\overline{\C_{+}}$,  so \eqref{eq:abstractpolas2} is well defined. Also, as in Proposition \ref{prop:qtoinf}, the specific value of $\beta$ in \eqref{eq:abstractpolas11} plays no role.

\begin{proof}
We want to show that $\|T(t)x\|_{X}\leq C_{n}t^{-n}\|x\|_{D_{A}(\gamma+s,p)}$ for all $t\geq1$ and $x\in D_{A}(\gamma+s,p)$. Since $D(A^{l})$ is dense in $D_{A}(\gamma+s,p)$ whenever $l\in\N$ satisfies $l>\gamma+s$, we may suppose throughout that $x\in D(A^{l})$ for some large $l$.
Hence, setting $g(t):=t^n \ind_{(0,\infty)}(t)T(|t|)x $ for $t\in \R$, by \cite[Proposition 4.3]{Rozendaal-Veraar18a} we may suppose that $g\in L^{1}(\R;X)$. In turn, \cite[Lemma 3.1]{Rozendaal-Veraar18a} then implies that $\widehat {g}(\xi)=n!R(i\xi,A)^{n+1}x$ for all $\xi\in\R$.

Next, note that $(T(t))_{t\geq0}$ restricts to a $C_{0}$-semigroup on $D_A(\gamma,p)$, the generator of which is the part of $A$ in $D_{A}(\gamma,p)$, which has domain $D_{A}(\gamma+1,p)$. In particular, we may fix $M\geq1$ and $\w\in\R$ such that $\|T(t)\|_{\La(D_A(\gamma,p))}\leq Me^{(\w-1) t}$ for all $t\geq0$. Set $f(t):=\ind_{(0,\infty)}(t)e^{-\omega t}T(|t|) x$. Then
\begin{equation}\label{eq:fbound1}
\|f\|_{L^{\infty}(\R;D_A(\gamma,p))}\leq M\|x\|_{D_A(\gamma,p)}\lesssim \|x\|_{D_A(\gamma+s,p)}.
\end{equation}
Moreover,
\begin{equation}\label{eq:fbound2}
\|f\|_{B^s_{p,p}(\R;D_A(\gamma,p))}\lesssim\|x\|_{(D_A(\gamma,p),D_A(\gamma+1,p))_{s,p}}\lesssim \|x\|_{D_A(\gamma+s,p)},
\end{equation}
as follows from Proposition \ref{prop:interpsemigroup} and \cite[Theorem 1.10.2]{Triebel95}. Also, again by \cite[Lemma 3.1]{Rozendaal-Veraar18a}, $\widehat{f}(\xi)=R(\omega+i\xi,A)x$ for all $\xi\in\R$. In particular, if we set
\[
m(\xi):=n!(R(i\xi,A)^{n}+\omega R(i\xi,A)^{n+1}),
\]
then $m(\xi)\widehat{f}(\xi)=\widehat {g}(\xi)$.

By combining all this, we see that
\begin{equation}\label{eq:multisplit}
\begin{aligned}
\sup_{t\geq 0}\|t^n T(t)x\|_{X}&=\|T_{m}f\|_{L^{\infty}(\R;X)}\\
&\leq n!\big(\|T_{R(i\cdot,A)^{n}}f\|_{L^{\infty}(\R;X)}+\w\|T_{R(i\cdot,A)^{n+1}}f\|_{L^{\infty}(\R;X)}\big).
\end{aligned}
\end{equation}
For $n>0$, one can apply \eqref{eq:abstractpolas2} and Proposition \ref{prop:qtoinf} to the final line, and then use \eqref{eq:fbound2} as well, to obtain
\[
\sup_{t\geq 0}\|t^n T(t)x\|_X\lesssim\|f\|_{B^s_{p,p}(\R;D_A(\gamma,p))}\lesssim \|x\|_{D_A(\gamma+s,p)}.
\]
For $n=0$, the same reasoning can be used for the second term in brackets in \eqref{eq:multisplit}, while for the first term one can directly rely on \eqref{eq:fbound1}, since $T_{R(i\cdot,A)^{0}}f=f$.
\end{proof}

We are now ready to prove the first part of Theorem \ref{thm:main}.

\begin{theorem}\label{thm:polinter}
Let $A$ be the generator of a $C_{0}$-semigroup $(T(t))_{t\geq0}$ on a Banach space $X$ with Fourier type $p\in[1,2]$. Suppose that $\C_{+}\subseteq\rho(A)$, and that there exist 
$\beta>0$ and $C\geq0$ such that
\[
\|R(\lambda,A)\|_{\La(X)} \leq C(1+|\lambda|)^{\beta}\quad(\lambda\in\C_{+}).\\
\]
Let $\rho\geq0$ and set $\tau:=(\rho+1)\beta+\tfrac{1}{p}-\tfrac{1}{p'}$. Then there exists a $C_{\rho}\geq0$ such that
\begin{equation}\label{eq:polinter}
\|T(t)x\|_{X}\leq C_{\rho}t^{-\rho}\|x\|_{D_A(\tau,p)}
\end{equation}
for all $t\geq1$ and $x\in D_A(\tau,p)$.
Moreover, if $\rho>0$, then \eqref{eq:polinter} also holds with $D_A(\tau,p)$ replaced by $D_A(\tau,q)$ for any $q\in [1, \infty]$, or by $D((-A)^{\tau})$.
\end{theorem}
\begin{proof}
We first consider the case where $\rho\in\N_{0}$. 
Recall that, since $X$ has Fourier type $p$, so does 
$D_{A}((\rho+1)\beta,p)$.
Moreover, by Proposition \ref{prop:resdecay},
\[
\sup_{\xi\in\R}\|R(i\xi,A)^k\|_{\La(D_{A}((\rho+1)\beta,p),X)}<\infty
\]
for all $k\in\{0,\ldots, \rho+1\}$. Hence Proposition \ref{prop:multFourtype} implies that
\[
T_{R(i\cdot,A)^k}:B^{1/p-1/p'}_{p,p}(\R;D_{A}((\rho+1)\beta,p))\to L^{p'}(\R;X)
\]
is bounded for every $k\in\{0,\ldots, \rho+1\}$. Finally, Theorem \ref{thm:abstractpol} yields \eqref{eq:polinter}.

To extend \eqref{eq:polinter} to general $\rho\geq0$ we proceed as follows. Fix $t\geq 1$ and $\rho>0$. Let $\rho_0, \rho_1\in\N_{0}$ be such that $\rho_0<\rho<\rho_1$, and let $\theta\in (0,1)$ be such that $\rho = (1-\theta)\rho_0 + \theta \rho_1$. Set $\tau_i := (\rho_i+1)\beta+\tfrac{1}{p}-\tfrac{1}{p'}$ for $i\in\{0,1\}$. Then, by what we have already shown,
\[
\|T(t)\|_{\La(D_A(\tau_i,p),X)}\leq C_{\rho_i} t^{-\rho_i}
\]
for some constant $C_{\rho_{i}}\geq0$ independent of $t$. Now, due to reiteration, real interpolation with parameters $\theta$ and $q\in[1,\infty]$ gives
\[
\|T(t)\|_{\La(D_A(\tau,q),X)}\leq C_{\rho}t^{-\rho}
\]
for some $C_{\rho}\geq0$ independent of $t$. This proves both \eqref{eq:polinter} and the final statement of the theorem, since $D((-A)^{\tau})\subseteq D_{A}(\tau,\infty)$.
\end{proof}

\section{Polynomial stability on fractional domains}\label{sec:stabfrac}

This section is devoted to the proof of the final statement in Theorem \ref{thm:main}.

The following proposition will play the same role in this section that Proposition \ref{prop:multFourtype} did in the previous section. For $\gamma\in\R$, recall the definition of the weight $w_{\gamma}:\R\to[0,\infty)$ from \eqref{eq:weight}.

\begin{proposition}\label{prop:multiplierHLtype} Let $p,q\in[1,\infty]$, $r\in[1,\infty)$, $\gamma\in\R$ and $\delta\in(-\infty,1/r')$.
Let $Y$ be a Banach space such that
\begin{equation}\label{eq:HLassum1}
\F:L^{p}(\R;Y)\to L^{r}(\R,w_{\gamma r};Y)
\end{equation}
is bounded, and let $X$ be a Banach space such that
\begin{equation}\label{eq:HLassum2}
\F:L^{r}(\R,w_{\delta r};X)\to L^{q}(\R;X)
\end{equation}
is bounded.
Let $m:\R\setminus\{0\}\to\La(Y,X)$ be an $X$-strongly measurable map for which there exists a $C\geq0$ such that $\norm{m(\xi)}_{\La(Y,X)}\leq C |\xi|^{\gamma-\delta}$ for all $\xi\in\R\setminus\{0\}$. Then $T_m:L^p(\R;Y)\to L^{q}(\R;X)$ is bounded.
\end{proposition}
\begin{proof}
Simply combine the assumptions on $X$, $m$ and $Y$:
\begin{align*}
\|T_m f\|_{L^{q}(\R;X)}\lesssim\|m \wh{f}\,\|_{L^{r}(\R,w_{\delta r};X)}\lesssim \|w_{\delta}w_{\gamma-\delta} \wh{f}\,\|_{L^{r}(\R;Y)}\lesssim \|f\|_{L^{p}(\R;Y)}
\end{align*}
for all $f\in L^{p}(\R;Y)$. Note that $L^{r}(\R,w_{\delta r};X)\subseteq\Sw'(\R;X)$, since $\delta<1/r'$.
\end{proof}

We are now ready to prove the last statement in Theorem \ref{thm:main}, as a special case of the following result.

\begin{theorem}\label{thm:polydecayHL}
Let $A$ be the generator of a $C_{0}$-semigroup $(T(t))_{t\geq0}$ on a Banach space $X$ with Fourier type $p\in(1,2]$, and suppose that $X$ has Hardy--Littlewood type $p$ or Hardy--Littlewood cotype $p'$. Suppose that $\C_{+}\subseteq\rho(A)$, and that there exist 
$\beta>0$ and $C\geq0$ such that
\[
\|R(\lambda,A)\|_{\La(X)} \leq C(1+|\lambda|)^{\beta}\quad(\lambda\in\C_{+}).
\]
Let $\rho\geq0$ and set $\tau:=(\rho+1)\beta+\tfrac{1}{p}-\tfrac{1}{p'}$. Then there exists a $C_{\rho}\geq0$ such that
\[
\|T(t)x\|_{X}\leq C_{\rho}t^{-\rho}\|x\|_{D((-A)^{\tau})}
\]
for all $t\geq 1$ and $x\in D((-A)^{\tau})$.
\end{theorem}

Note that Theorem \ref{thm:polydecayHL} is independent of Theorem \ref{thm:polinter} in the special case where $\rho=0$. For $\rho>0$ the conclusion of Theorem \ref{thm:polydecayHL} already follows from Theorem \ref{thm:polinter}.

\begin{proof}
The proof is analogous to that of \cite[Theorem 4.9]{Rozendaal-Veraar18a}, and as such also similar to the proof of Theorem \ref{thm:polinter}. We will indicate what the key steps are.

By interpolation, cf.~\cite[Lemma 4.2]{Rozendaal-Veraar18a}, it suffices to consider $\rho\in\N_{0}$. Then, using arguments as before but relying on \cite[Proposition 3.4]{Rozendaal-Veraar18a} instead of Proposition \ref{prop:resdecay}, one can show that
\begin{equation}\label{eq:polydecaymult}
\|R(i\xi,A)^{k}\|_{\La(D((-A)^{\tau}),X)}\lesssim (1+|\xi|)^{-(\frac{1}{p}-\frac{1}{p'})}
\end{equation}
for all $k\in\{1,\ldots,\rho+1\}$ and an implicit constant independent of $\xi\in\R$.

Next, note that $D((-A)^{\tau})$ has the same Fourier type and Hardy--Littlewood type and cotype as $X$, because $X$ and $D((-A)^{\tau})$ are isomorphic. In particular, if $X$ has Hardy--Littlewood type $p$, then one may apply Proposition \ref{prop:multiplierHLtype} with $r=p$, $q=p'$, $\gamma=\tfrac{1}{p'}-\tfrac{1}{p}$ and $\delta=0$. On the other hand, if $X$ has Hardy--Littlewood cotype $p'$, then one can apply Proposition \ref{prop:multiplierHLtype} with $q=r=p'$, $\gamma=0$ and $\delta=\tfrac{1}{p}-\tfrac{1}{p'}$. In both cases, it follows from \eqref{eq:polydecaymult} that
\[
T_{R(i\cdot,A)^{k}}:L^{p}(\R;D((-A)^{\tau}))\to L^{p'}(\R;X)
\]
 is bounded for all $k\in\{1,\ldots,\rho+1\}$. Now \cite[Theorem 4.6]{Rozendaal-Veraar18a} concludes the proof.
\end{proof}

\section*{Acknowledgments}\label{sec:acknowledge}

The second author is grateful to Yuri Tomilov for helpful conversations about the topic of this article. The third author would like to thank Oscar Dominguez for pointing out reference \cite{GaToKa96} and for helpful discussions.

\section*{Data sharing and conflict of interest statement.} Data sharing not applicable to this article as no datasets were generated or analysed during the current study.

The authors have no competing interests to declare that are relevant to the content of this article.

\bibliographystyle{plain}

\bibliography{Bibliographystab}

\begin{thebibliography}{10}

\bibitem{Anantharaman-Leautaud14}
Nalini Anantharaman and Matthieu L{\'e}autaud.
\newblock Sharp polynomial decay rates for the damped wave equation on the
  torus.
\newblock {\em Anal. PDE}, 7(1):159--214, 2014.
\newblock With an appendix by St{\'e}phane Nonnenmacher.

\bibitem{ArBaHiNe11}
Wolfgang Arendt, Charles J.~K. Batty, Matthias Hieber, and Frank Neubrander.
\newblock {\em Vector-valued {L}aplace transforms and {C}auchy problems},
  volume~96 of {\em Monographs in Mathematics}.
\newblock Birkh\"auser/Springer Basel AG, Basel, second edition, 2011.

\bibitem{BaEnPrSc06}
Andr\'as B\'atkai, Klaus-Jochen Engel, Jan Pr\"uss, and Roland Schnaubelt.
\newblock Polynomial stability of operator semigroups.
\newblock {\em Math. Nachr.}, 279(13-14):1425--1440, 2006.

\bibitem{BaChTo16}
Charles J.~K. Batty, Ralph Chill, and Yuri Tomilov.
\newblock Fine scales of decay of operator semigroups.
\newblock {\em J. Eur. Math. Soc. (JEMS)}, 18(4):853--929, 2016.

\bibitem{Batty-Duyckaerts08}
Charles J.~K. Batty and Thomas Duyckaerts.
\newblock Non-uniform stability for bounded semi-groups on {B}anach spaces.
\newblock {\em J. Evol. Equ.}, 8(4):765--780, 2008.

\bibitem{Borichev-Tomilov10}
Alexander Borichev and Yuri Tomilov.
\newblock Optimal polynomial decay of functions and operator semigroups.
\newblock {\em Math. Ann.}, 347(2):455--478, 2010.

\bibitem{Burq98}
Nicolas Burq.
\newblock D\'ecroissance de l'\'energie locale de l'\'equation des ondes pour
  le probl\`eme ext\'erieur et absence de r\'esonance au voisinage du r\'eel.
\newblock {\em Acta Math.}, 180(1):1--29, 1998.

\bibitem{ChSeTo20}
Ralph Chill, David Seifert, and Yuri Tomilov.
\newblock Semi-uniform stability of operator semigroups and energy decay of
  damped waves.
\newblock {\em Philos. Trans. Roy. Soc. A}, 378(2185):20190614, 24, 2020.

\bibitem{Cordero-Nicola09}
Elena Cordero and Fabio Nicola.
\newblock Sharpness of some properties of {W}iener amalgam and modulation
  spaces.
\newblock {\em Bull. Aust. Math. Soc.}, 80(1):105--116, 2009.

\bibitem{Davies-Simon91}
E.~Brian Davies and Barry Simon.
\newblock {$L^p$} norms of noncritical {S}chr\"odinger semigroups.
\newblock {\em J. Funct. Anal.}, 102(1):95--115, 1991.

\bibitem{Dominguez-Veraar21}
Oscar Dominguez and Mark Veraar.
\newblock Extensions of the vector-valued {H}ausdorff-{Y}oung inequalities.
\newblock {\em Math. Z.}, 299(1-2):373--425, 2021.

\bibitem{GaToKa96}
Jos\'{e} Garc{\'{\i}}a-Cuerva, Jos\'{e}~L. Torrea, and Kazaros~S. Kazarian.
\newblock On the {F}ourier type of {B}anach lattices.
\newblock In {\em Interaction between functional analysis, harmonic analysis,
  and probability ({C}olumbia, {MO}, 1994)}, volume 175 of {\em Lecture Notes
  in Pure and Appl. Math.}, pages 169--179. Dekker, New York, 1996.

\bibitem{Gearhart78}
Larry Gearhart.
\newblock Spectral theory for contraction semigroups on {H}ilbert space.
\newblock {\em Trans. Amer. Math. Soc.}, 236:385--394, 1978.

\bibitem{Goldstein-Wacker03}
Jerome~A. Goldstein and Markus Wacker.
\newblock The energy space and norm growth for abstract wave equations.
\newblock {\em Appl. Math. Lett.}, 16(5):767--772, 2003.

\bibitem{Grigoryan06}
Alexander Grigor'yan.
\newblock Heat kernels on weighted manifolds and applications.
\newblock In {\em The ubiquitous heat kernel}, volume 398 of {\em Contemp.
  Math.}, pages 93--191. Amer. Math. Soc., Providence, RI, 2006.

\bibitem{Haase06a}
Markus Haase.
\newblock {\em The functional calculus for sectorial operators}, volume 169 of
  {\em Operator Theory: Advances and Applications}.
\newblock Birkh\"auser Verlag, Basel, 2006.

\bibitem{Huang85}
Fa~Lun Huang.
\newblock Characteristic conditions for exponential stability of linear
  dynamical systems in {H}ilbert spaces.
\newblock {\em Ann. Differential Equations}, 1(1):43--56, 1985.

\bibitem{HyNeVeWe16}
Tuomas Hyt\"{o}nen, Jan van Neerven, Mark Veraar, and Lutz Weis.
\newblock {\em Analysis in {B}anach spaces. {V}ol. {I}. {M}artingales and
  {L}ittlewood-{P}aley theory}, volume~63 of {\em Ergebnisse der Mathematik und
  ihrer Grenzgebiete. 3. Folge. A Series of Modern Surveys in Mathematics
  [Results in Mathematics and Related Areas. 3rd Series. A Series of Modern
  Surveys in Mathematics]}.
\newblock Springer, Cham, 2016.

\bibitem{HyNeVeWe23}
Tuomas Hyt\"onen, Jan van Neerven, Mark Veraar, and Lutz Weis.
\newblock {\em Analysis in {B}anach spaces. {V}ol. {III}. {H}armonic analysis
  and spectral theory}, volume~76 of {\em Ergebnisse der Mathematik und ihrer
  Grenzgebiete. 3. Folge. A Series of Modern Surveys in Mathematics [Results in
  Mathematics and Related Areas. 3rd Series. A Series of Modern Surveys in
  Mathematics]}.
\newblock Springer, Cham, 2023.

\bibitem{Kwapien72}
Stanis{\l}aw Kwapie{\'n}.
\newblock Isomorphic characterizations of inner product spaces by orthogonal
  series with vector valued coefficients.
\newblock {\em Studia Math.}, 44:583--595, 1972.
\newblock Collection of articles honoring the completion by Antoni Zygmund of
  50 years of scientific activity, VI.

\bibitem{Lebeau96}
Gilles Lebeau.
\newblock \'{E}quation des ondes amorties.
\newblock In {\em Algebraic and geometric methods in mathematical physics
  ({K}aciveli, 1993)}, volume~19 of {\em Math. Phys. Stud.}, pages 73--109.
  Kluwer Acad. Publ., Dordrecht, 1996.

\bibitem{Paunonen14}
Lassi Paunonen.
\newblock Polynomial stability of semigroups generated by operator matrices.
\newblock {\em J. Evol. Equ.}, 14(4-5):885--911, 2014.

\bibitem{Pruss84}
Jan Pr\"{u}ss.
\newblock On the spectrum of {$C_{0}$}-semigroups.
\newblock {\em Trans. Amer. Math. Soc.}, 284(2):847--857, 1984.

\bibitem{Ralston69}
James~V. Ralston.
\newblock Solutions of the wave equation with localized energy.
\newblock {\em Comm. Pure Appl. Math.}, 22:807--823, 1969.

\bibitem{Rauch-Taylor74}
Jeffrey Rauch and Michael Taylor.
\newblock Exponential decay of solutions to hyperbolic equations in bounded
  domains.
\newblock {\em Indiana Univ. Math. J.}, 24:79--86, 1974.

\bibitem{Renardy94}
Michael Renardy.
\newblock On the linear stability of hyperbolic {PDE}s and viscoelastic flows.
\newblock {\em Z. Angew. Math. Phys.}, 45(6):854--865, 1994.

\bibitem{Rozendaal23}
Jan Rozendaal.
\newblock Operator-valued {$(L^p,L^q)$} {F}ourier multipliers and stability
  theory for evolution equations.
\newblock {\em Indag. Math. (N.S.)}, 34(1):1--36, 2023.

\bibitem{Rozendaal-Schippa23}
Jan Rozendaal and Robert Schippa.
\newblock Nonlinear wave equations with slowly decaying initial data.
\newblock {\em J. Differential Equations}, 350:152--188, 2023.

\bibitem{RoSeSt19}
Jan Rozendaal, David Seifert, and Reinhard Stahn.
\newblock Optimal rates of decay for operator semigroups on {H}ilbert spaces.
\newblock {\em Adv. Math.}, 346:359--388, 2019.

\bibitem{Rozendaal-Veraar18b}
Jan Rozendaal and Mark Veraar.
\newblock Sharp growth rates for semigroups using resolvent bounds.
\newblock {\em J. Evol. Equ.}, 18(4):1721--1744, 2018.

\bibitem{Rozendaal-Veraar18a}
Jan Rozendaal and Mark Veraar.
\newblock Stability theory for semigroups using ({$L^p, L^q$}) {F}ourier
  multipliers.
\newblock {\em J. Funct. Anal.}, 275(10):2845--2894, 2018.

\bibitem{Santana-Carvalho24}
Genilson Santana and Silas~L. Carvalho.
\newblock Refined decay rates of {$C_0$}-semigroups on {B}anach spaces.
\newblock {\em J. Evol. Equ.}, 24(2):Paper No. 28, 59, 2024.

\bibitem{Sklyar-Polak17}
Grigory~M. Sklyar and Piotr Polak.
\newblock On asymptotic estimation of a discrete type {$C_0$}-semigroups on
  dense sets: application to neutral type systems.
\newblock {\em Appl. Math. Optim.}, 75(2):175--192, 2017.

\bibitem{Triebel95}
Hans Triebel.
\newblock {\em Interpolation theory, function spaces, differential operators}.
\newblock Johann Ambrosius Barth, Heidelberg, second edition, 1995.

\bibitem{vanNeerven09}
Jan van Neerven.
\newblock Asymptotic behaviour of {$C_0$}-semigroups and {$\gamma$}-boundedness
  of the resolvent.
\newblock {\em J. Math. Anal. Appl.}, 358(2):380--388, 2009.

\bibitem{Weis95}
Lutz Weis.
\newblock The stability of positive semigroups on {$L_p$} spaces.
\newblock {\em Proc. Amer. Math. Soc.}, 123(10):3089--3094, 1995.

\bibitem{Weis-Wrobel96}
Lutz Weis and Volker Wrobel.
\newblock Asymptotic behavior of {$C_0$}-semigroups in {B}anach spaces.
\newblock {\em Proc. Amer. Math. Soc.}, 124(12):3663--3671, 1996.

\bibitem{Wrobel89}
Volker Wrobel.
\newblock Asymptotic behavior of {$C_0$}-semigroups in {$B$}-convex spaces.
\newblock {\em Indiana Univ. Math. J.}, 38(1):101--114, 1989.

\end{thebibliography}

\end{document}